\documentclass{amsart}
\usepackage{graphicx}
\usepackage{epstopdf}
\usepackage{amsmath}
\usepackage{amsthm}
\usepackage{amssymb}

\newcommand{\comment}[1]{}

\newtheorem{theorem}{Theorem}[section]
\newtheorem{lemma}[theorem]{Lemma}
\newtheorem{corollary}[theorem]{Corollary}
\newtheorem{definition}[theorem]{Definition}
\newtheorem{proposition}[theorem]{Proposition}
\newtheorem*{BNthm}{Brill-Noether Theorem}
\newtheorem*{GPthm}{Gieseker-Petri Theorem}
\newcommand{\Spec}{\mathrm{Spec}\,}

\begin{document}

\title{Towards a Tropical Proof of the Gieseker-Petri Theorem}
\author{Vyassa Baratham}
\author{David Jensen}
\author{Cristina Mata}
\author{Dat Nguyen}
\author{Shalin Parekh}
\date{}
\bibliographystyle{alpha}

\maketitle

\begin{abstract}
We use tropical techniques to prove a case of the Gieseker-Petri Theorem.  Specifically, we show that the general curve of arbitrary genus does not admit a Gieseker-Petri special pencil.
\end{abstract}

\section{Introduction}

A central object in the study of algebraic curves is the variety of linear series on a curve.  Given a smooth projective curve $X$, we write $\mathcal{G}^r_d (X)$ for the variety parameterizing linear series of degree $d$ and rank $r$ on $X$.  The nature of this variety for general curves is the central focus of two of the most celebrated theorems in modern algebraic geometry.

\begin{BNthm} \cite{GH}
If $X$ is a general curve of genus $g$, then dim $\mathcal{G}^r_d (X) = \rho = g-(r+1)(g-d+r)$.  If $\rho < 0$, then $\mathcal{G}^r_d (X)$ is empty.
\end{BNthm}

\begin{GPthm} \cite{Gieseker-Petri}
If $X$ is a general curve, then $\mathcal{G}^r_d (X)$ is smooth.
\end{GPthm}

These theorems differ from more classical results such as Riemann-Roch in that they concern \emph{general}, rather than arbitrary, curves.  As such, the original proofs due to Griffiths-Harris \cite{GH} and Gieseker \cite{Gieseker-Petri} make use of degeneration techniques.  These ideas were later refined by Eisenbud and Harris \cite{EHPetri}, giving a second proof of both theorems.  A subsequent proof, due to Lazarsfeld, avoids using degeneration arguments by working instead with curves on a K3 surface \cite{Laz}.

More recently, a team consisting of Cools, Draisma, Payne and Robeva provided an independent proof of the Brill-Noether Theorem using techniques from tropical geometry \cite{CDPR}.  More specifically, they use the theory of divisors on metric graphs, as developed by Baker and Norine in \cite{BakerNorine}, to construct a Brill-Noether general graph $\Gamma_g$ with first Betti number $g$.  Combining this with Baker's Specialization Lemma \cite{Baker}, which says that the rank of a divisor on a smooth curve over a discretely valued field jumps under specialization to the dual graph of the central fiber, they obtain a new proof of the Brill-Noether Theorem.

In this paper, we prove the $r=1$ case of the Gieseker-Petri Theorem using a similar approach.  In other words, we show that $\mathcal{G}^1_d (X)$ is smooth for the general curve $X$ of arbitrary genus.  To do this, we use the same metric graph $\Gamma_g$ that appears in \cite{CDPR}.  This graph, depicted below, consists of $g$ loops arranged in a chain.  Throughout, we assume that this graph has generic edge lengths -- specifically, that the ratio $\frac{\ell_i}{m_i}$ for each $i$ is not equal to the ratio of two positive integers whose sum is less than or equal to $2g-2$.  Given this, we prove the following:

\begin{theorem}
\label{mainthm}
The graph $\Gamma_g$ does not admit a positive-rank divisor $D$ such that $K_{\Gamma_g} - 2D$ is linearly equivalent to an effective divisor.
\end{theorem}

\begin{figure}[!h]
\centering
\includegraphics[width=1\textwidth]{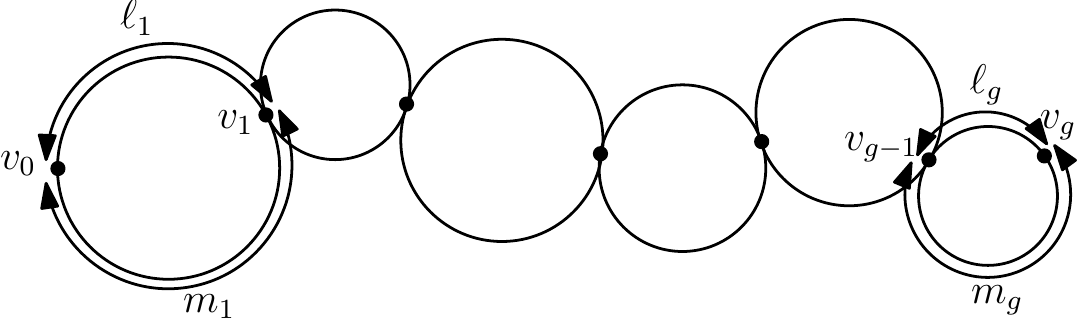}
\caption{The graph $\Gamma_g$ from \cite{CDPR}}	
\label{TheGraph}
\end{figure}

We show in Proposition \ref{Smoothness} that the above theorem implies the rank one case of the Gieseker-Petri Theorem.  In particular, we interpret the smoothness of $\mathcal{G}^1_d (X)$ at a basepoint-free pencil $W \subset H^0 (X,L)$ as a vanishing condition on $H^0 (X, K_X-2L)$.  We note that, for higher-rank linear series, the corresponding vanishing condition concerns a certain Koszul cohomology group rather than a space of global sections.  It would be interesting to know whether it's possible to detect the vanishing of Koszul cohomology groups using tropical techniques.  Such a theory, if developed, could potentially be used not only to provide tropical proofs of known theorems, such as the higher-rank cases of the Gieseker-Petri Theorem or Green's Conjecture for the general curve (see \cite{Voisin1} and \cite{Voisin2}), but also to shed light on open questions like the Maximal Rank Conjecture.

Our result, together with that of \cite{CDPR}, provides strong evidence that the graph $\Gamma_g$ is Gieseker-Petri general in the sense that $\mathcal{G}^r_d (X)$ is smooth for any curve $X$ that specializes in a regular family to a curve with dual graph $\Gamma_g$.  We mention one other piece of evidence in support of this.  In the case that $\rho = 0$, the variety $\mathcal{G}^r_d (X)$ is zero-dimensional, and the Gieseker-Petri theorem simply says that it is reduced.  This latter fact follows from \cite{CDPR}, where it is shown that the graph $\Gamma_g$ admits precisely $\lambda$ distinct divisors of degree $d$ and rank $r$, where $\lambda$ is the $r$-dimensional Catalan number
$$ \lambda := g! \prod_{k=0}^r \frac{k!}{(g-d+r+k)!} .$$

This paper is broken into three sections.  In the next section, we discuss the basic theory of linear systems on metric graphs.  In the third and final section, we use this theory to prove Theorem \ref{mainthm}.

\textbf{Acknowledgements:}
We would like to thank Eric Katz for reading an early version of this paper.

\section{Preliminaries}

This section contains a short outline of the facts we will need concerning divisors on metric graphs.  The full theory is developed in \cite{BakerNorine} and \cite{Baker}, which we encourage the reader to consult for more details.

\subsection{Divisors and Equivalence}

Given a metric graph $\Gamma$, we define the group $Div ( \Gamma )$ of divisors on $\Gamma$ to be the free abelian group on the points of $\Gamma$.  Given a divisor $D = \sum a_i p_i \in Div ( \Gamma )$, we define the \textit{degree} of $D$ to be the sum $\sum a_i$, and we say that $D$ is \textit{effective} if all of the coefficients $a_i$ are nonnegative.

In the tropical world, the role of meromorphic functions on an algebraic curve is played by piecewise linear functions on a metric graph.  More precisely, given a finite subdivision of $\Gamma$ and a continuous function $\psi$ on $\Gamma$ whose restriction to each edge of the subdivision is given by a linear function with integer slope, we define $ord_p ( \psi )$ to be the sum of the incoming slopes of $\psi$ along edges containing the point $p \in \Gamma$.  A \textit{principal} divisor on $\Gamma$ is then any divisor of the form
$$ div ( \psi ) := \sum_{p \in \Gamma} ord_p ( \psi ) p $$
for some piecewise linear function $\psi$.  In analogy with the case of algebraic curves, we define the Picard group $Pic ( \Gamma )$ to be the quotient of $Div ( \Gamma )$ by the subgroup of principal divisors.  We say that two divisors $D$ and $D'$ are equivalent, and write $D \sim D'$, if $D-D'$ is a principal divisor.

It is standard practice in combinatorics to refer to divisors on $\Gamma$ as \textit{chip configurations}.  In this language, a divisor $D = \sum a_i p_i$ is represented by a stack of $a_i$ chips at each point $p_i$ of the graph.  We will naturally turn to this language in our proof of Theorem \ref{mainthm}.

\subsection{Ranks of Divisors and Baker's Specialization Lemma}

Given a divisor $D$ on $\Gamma$, we say that $D$ has \textit{rank} $r$ if $r$ is the greatest integer such that $D-E$ is equivalent to an effective divisor for every effective divisor $E$ of degree $r$.  Throughout, we will say that a divisor \textit{moves} if it has positive rank.

Perhaps the most important property of divisors on metric graphs is their relation to divisors on algebraic curves.  Let $R$ be a DVR with field of fractions $K$ and residue field $k$, and let $X$ be a smooth projective curve over $K$.  A strongly semistable regular model of $X$ is a regular scheme $\mathcal{X}$ over $\Spec R$ whose general fiber is $X$ and whose special fiber is a reduced union of geometrically irreducible smooth curves meeting in nodes defined over $k$.  Let $\Gamma$ denote the metric graph corresponding to the dual graph of $X_k$, where every edge is assigned length 1.

Each point of $X(K)$ specializes to a smooth point of the special fiber, and hence is associated to a well-defined vertex of $\Gamma$.  Note that, if every component of the central fiber is rational, then the degree of the relative dualizing sheaf on each component is two less than the number of nodes, and hence the canonical divisor on $X$ specializes to
$$ K_{\Gamma} := \sum_{v \in \Gamma} (deg(v)-2) v .$$

If $K'$ is a finite extension of $K$, the variety $\mathcal{X} \times_K K'$ may not be a strongly semistable regular model of $X \times_K K'$, but this issue may be resolved by blowing up the singularities of the central fiber.  The dual graph of the special fiber of this new model is isomorphic to $\Gamma$, but with edges subdivided into $e$ segments, where $e$ is the ramification index of $K'$ over $K$.  Hence there is a well-defined map from the $\bar{K}$-points of $X$ to the points of $\Gamma$.  Extending this linearly defines a map on divisors $\tau_* : Div(X_{\bar{K}}) \to Div ( \Gamma )$.  Moreover, this map respects linear equivalence, and hence defines a map $\tau_* : Pic (X_{\bar{K}}) \to Pic ( \Gamma )$.  The key point of this construction is Baker's Specialization Lemma, which says that ranks of divisors are well-behaved under this map.

\begin{lemma}
\cite{Baker}
\label{BakerSpecialization}
Let $D$ be a divisor on $X_{\bar{K}}$.  Then $r( \tau_* (D)) \geq r(D)$.
\end{lemma}

\begin{proposition}
\label{Smoothness}
Let $\mathcal{X}$ be a strongly semistable regular model with general fiber $X$, and suppose that the central fiber has dual graph $\Gamma_g$.  Then Theorem \ref{mainthm} implies that $\mathcal{G}^1_d (X)$ is smooth.
\end{proposition}

\begin{proof}
Let $L$ be a line bundle on $X$ and $W \subset H^0 (X,L)$ be a 2-dimensional vector space.  By Proposition 4.1 in \cite{ACGH}, it suffices to show that the cup-product map
$$ \mu_W : W \otimes H^0 (X,K_X-L) \to H^0 (X,K_X) $$
is injective.  If $B$ is the base locus of $W$, then by the basepoint-free pencil trick (see p. 126 in \cite{ACGH}), we see that $\ker ( \mu_W ) \cong H^0 (X, K_X-2L+B)$.  Letting $L' = L-B$, it therefore suffices to show that $K_X-2L'$ is not effective.  By Lemma \ref{BakerSpecialization}, we therefore have
$$ r( \tau_* L' ) \geq r(L') \geq 1, $$
$$ r( K_{\Gamma_g} - 2 \tau_* L' ) = r( \tau_* (K_X - 2L')) \geq r(K_X - 2L') .$$
By Theorem \ref{mainthm}, however, since $\tau_* L'$ has positive rank, $K_{\Gamma_g} - 2 \tau^* L'$ is not linearly equivalent to an effective divisor, and hence has negative rank.  The result follows.
\end{proof}

\subsection{Reduced Divisors and Lingering Lattice Paths}

A useful tool for working with divisors on metric graphs is the notion of reduced divisors.  For a fixed point $p \in \Gamma$, we say that an effective divisor $D$ is $p$-\textit{reduced} if the set of distances from $p$ to chips of $D$ is lexicographically minimal among all effective divisors equivalent to $D$.  By definition, every effective divisor is equivalent to a unique $p$-reduced divisor.

It is straightforward to characterize the $v_n$-reduced divisors on $\Gamma_g$, and indeed this is done in \cite{CDPR}.  Let $\bar{\gamma_j}$ denote the $j^{th}$ loop of $\Gamma_g$.  For each $j>n$, let $\gamma_j = \bar{\gamma_j} \backslash \{ v_{j-1} \}$ be the corresponding left-punctured loop, and for each $j \leq n$, let $\gamma_j' = \bar{\gamma_j} \backslash \{ v_{j} \}$ be the corresponding right-punctured loop, as pictured below.  An effective divisor $D$ is $v_n$-reduced if and only if each such cell contains at most one chip of $D$.

\begin{figure}[!h]
\centering
\includegraphics[width=1\textwidth]{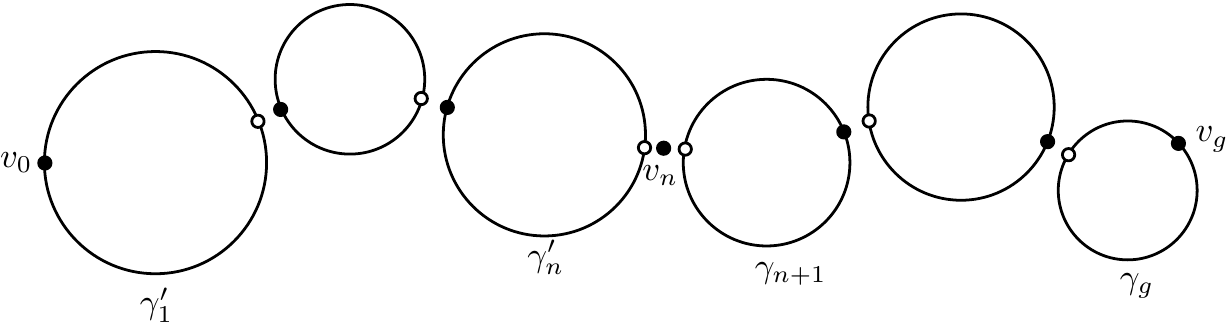}
\caption{Cell decomposition of the graph $\Gamma_g$ (from \cite{CDPR})}	
\label{fig:CellDecomposition}
\end{figure}

One of the main results of \cite{CDPR} is a characterization of those divisors on $\Gamma_g$ that have rank $r$.  Every effective divisor on $\Gamma_g$ is equivalent to a $v_0$-reduced divisor, and every such a divisor consists of $d_0$ chips at the vertex $v_0$, together with at most one chip on every other loop.  We may therefore associate to each equivalence class the data $(d_0 , x_1 , x_2 , \ldots x_g )$, where $x_i \in \mathbb{R} / (\ell_i + m_i ) \mathbb{Z}$ is the distance from the chip on the $i^{th}$ loop to $v_{i-1}$ in the counterclockwise direction.  (If there is no chip on the $i^{th}$ loop, we write $x_i = 0$.)  The associated lingering lattice path is defined as follows.

\begin{definition}
\cite{CDPR}
Let $D$ be the $v_0$-reduced divisor of degree $d$ corresponding to $(d_0 , x_1 , \ldots , x_g )$.  Then the associated \textit{lingering lattice path} $P$ in $\mathbb{Z}^r$ starts at $(d_0 , d_0 -1, \ldots , d_0 -r+1)$ with steps given by
\begin{displaymath}
p_i - p_{i-1} = \left\{ \begin{array}{ll}
(-1,-1, \ldots , -1) & \textrm{if $x_i = 0$}\\
e_j & \textrm{if $x_i = (p_{i-1}(j) +1)m$ mod $\ell_i + m_i$} \\
  & \textrm{and both $p_{i-1}$ and $p_{i-1} + e_j$ are in $\mathcal{C}$}\\
0 & \textrm{otherwise}
\end{array} \right\}
\end{displaymath}
where $e_0 , \ldots e_{r-1}$ are the standard basis vectors in $\mathbb{Z}^r$.
\end{definition}

They then prove:

\begin{proposition}
\cite{CDPR}
A divisor $D$ on $\Gamma_g$ has rank at least $r$ if and only if the associated lingering lattice path lies entirely in the open Weyl chamber
$$ \{ y \in \mathbb{Z}^r \vert y_0 > y_1 > \cdots > y_{r-1} > 0 \}. $$
\end{proposition}

In this paper, we are only interested in the case where $r=1$, in which case the lingering lattice path is simply a sequence of integers $p_i$.  The proposition above says that a given divisor $D$ moves if and only if $p_i > 0$ for all $i$.  It is shown in Proposition 3.10 of \cite{CDPR} that the $v_n$-reduced divisor equivalent to $D$ has precisely $p_n$ chips at $v_n$.

\section{Combinatorial Arguments}


In this section, we prove Theorem \ref{mainthm}.  Our approach is via induction on $g$.  The base cases $g = 1,2,3$ follow from [CDPR], as $\Gamma_g$ is Brill-Noether general.  Throughout, we will suppose that there exists a pair $(D,E)$ of divisors on $\Gamma_g$ such that $D$ moves, $E$ is effective, and $2D+E \sim K_{\Gamma_g}$.  We will furthermore write $(p_0, p_1, p_2, \ldots , p_g)$ for the (rank one) lingering lattice path associated to $D$.

Without loss of generality, we may assume that $D$ and $E$ are both $v_0$-reduced.  Indeed, letting $D'$ and $E'$ be the $v_0$-reduced divisors equivalent to $D$ and $E$, we see that $2D' + E' \sim 2D + E \sim K_{\Gamma_g}$.  Moreover, we may assume that $D$ is not supported at any of the vertices $v_k$ for $k \geq 1$.  To see this, let $v'_k$ be the point on $\bar{\gamma_k}$ a distance of $\frac{1}{2} (\ell_k + m_k)$ from $v_k$. Setting $D' = D + v'_k - v_k$, we see that $2D' \sim 2D$, and since the $k^\mathrm{th}$ step of the lingering lattice path associated to $D$ is lingering, $D'$ moves if and only if $D$ does.

By assumption, there exists a continuous piecewise linear function $\psi$ on $\Gamma_g$ such that $2D+E-K_{\Gamma_g} = div(\psi)$.  For each $k$, let $\psi_k$ be the restriction of $\psi$ to $\bar{\gamma_k}$.  The function $\psi$ plays a pivotal role in the arguments that follow, and our first task is to determine a few of its properties.

\subsection{Properties of the Function $\psi$}

We will use the following lemma and its corollaries.

\begin{lemma}
\label{lemma3}
Consider the subgraph $ \Gamma_{g-k}$ obtained by removing the right-punctured loops $\gamma_1' , \ldots , \gamma_k'$. For each $k \geq 1$, we have deg$(2D + E) \vert_{\Gamma_{g-k}} = ord_{v_k} (\psi_k) + 2(g-k)$.
\end{lemma}

\begin{proof}
Let $D'$ and $E'$ be the restriction of $D$ and $E$ to $\Gamma \backslash \Gamma_{g-k}$, and let $\psi'$ be the restriction of $\psi$ to $\Gamma_k := ( \Gamma \backslash \Gamma_{g-k} ) \cup \{ v_k \}$.  For all $p \in \Gamma_k$, we have
\[ord_p (\psi') = \left\{ \begin{array}{ll} ord_{v_k} (\psi_k) & \mbox{if $p = v_k$};\\ ord_p (\psi) & \mbox{otherwise.} \end{array} \right.\]
Therefore, $div(\psi') = 2D' + E' + ord_{v_k} (\psi_k) v_k - K_{\Gamma_k}$.  We have the following equations:
\begin{eqnarray*}
deg \left( 2D' + E' - K_{\Gamma_k} \right) + ord_{v_k} (\psi_k) = 0\\
deg \left( 2D + E - K_{\Gamma_g} \right) = 0\\
deg \left(K_{\Gamma_g} \right) - deg \left(K_{\Gamma_k} \right) = 2(g-k)\\
\end{eqnarray*}
It follows that $deg(2D+E) - deg (2D'+E') = ord_{v_k} (\psi_k) + 2(g-k)$.  By definition, the left hand side is precisely $deg(2D + E) \vert_{\Gamma_{g-k}}$.
\end{proof}

As a first consequence, we obtain bounds on the incoming slopes of $\psi_k$ at each of the vertices.

\begin{corollary}
\label{corollary1}
For all $k<g$, we have $2(k-g) \leq ord_{v_k} (\psi_k) < 0$.
\end{corollary}

\begin{proof}
The lefthand inequality follows directly from Lemma \ref{lemma3}.  To see the righthand inequality, note that in the proof of Lemma \ref{lemma3}, $2D' + E' + ord_{v_k} (\psi_k)v_k \sim K_{\Gamma_k}$. In addition, the lingering lattice path associated to $D'$ is $(p_0, p_1, \ldots , p_k)$, so $D'$ moves.  If $ord_{v_k} (\psi_k) \geq 0$, then $E' + ord_{v_k} (\psi_k)v_k$ is effective, and thus the pair $(D', E'+ord_{v_k} (\psi_k)v_k)$ on $\Gamma_k$ contradicts our inductive hypothesis.
\end{proof}

Next, we see that the possible distributions of chips of $D$ and $E$ is quite limited.

\begin{corollary}
\label{corollary3}
For all $k<g$, there is at least one chip of $D$ or $E$ on $\gamma_k \backslash \{ v_k \}$.
\end{corollary}
\begin{proof}
Suppose that there is some positive integer $k < g$ such that neither $D$ nor $E$ has a chip on $\gamma_k \backslash \{ v_k \}$. Let $x$ be the slope of $\psi$ on the longer edge of $\bar{\gamma_k}$. The slope of $\psi$ on the shorter edge of $\bar{\gamma_k}$ must then be $x + ord_{v_k} (\psi_k)$.

\begin{figure}[!h]
\centering
\includegraphics[width=.25\textwidth]{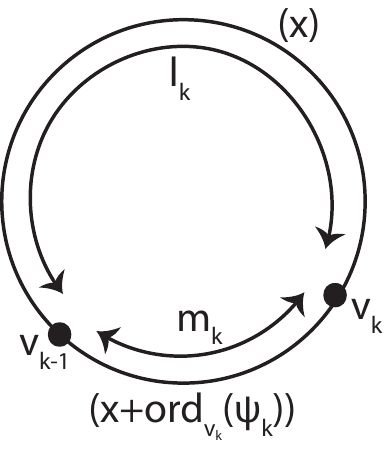}
\caption{The slopes  of $\psi$ (outside the loop) and distances (inside the loop) on $\bar{\gamma_k}$}	
\label{fig:Slopenochip}
\end{figure}

We have the following equation:
\begin{equation*}
x \ell_k + (x + ord_{v_k} (\psi_k))m_k = 0
\end{equation*}
Which gives
\begin{equation}
\label{eqn:lk/mk}
	\frac{\ell_k}{m_k} = \frac{x + ord_{v_k} (\psi_k)}{-x}
\end{equation}

The numerator and the denominator of the right hand side are integers which add up to $ord_{v_k} (\psi_k)$.  But, by Corollary \ref{corollary1}, we have $\vert ord_{v_k} (\psi_k) \vert \leq 2(g-k) \leq 2(g-1)$, contradicting the genericity condition on the edge lengths of $\Gamma_g$.
\end{proof}

\begin{corollary}
\label{corollary4}
Neither $D$ nor $E$ has a chip on $\gamma_g$.
\end{corollary}
\begin{proof}By corollary \ref{corollary1}, $ord_{v_{g-1}} (\psi_{g-1})\leq -1$. By Lemma \ref{lemma3}, it follows that $(2D + E)$ has at most $1$ chip on $\bar{\gamma_g}$, which means $D$ has no chip on $\gamma_g$.

Suppose that $E$ has a chip on $\gamma_g$. The point containing this chip and $v_{g-1}$ divide $\bar{\gamma_g}$ into two edges. Let $d_1, d_2$ denote their lengths, and $x_1, x_2$ denote the slopes of $\psi$ on these edges. We have $d_1 x_1 + d_2 x_2 = 0$. Since $d_1$ and $d_2$ are positive, this implies that either $x_1 = x_2 = 0$ or $x_1$ and  $x_2$ are of opposite sign. Given that $x_1$ and $x_2$ are two consecutive integers, neither case is possible.
\end{proof}

Finally, we see that the non-lingering steps in the lingering lattice path determine the incoming slopes of $\psi_k$ precisely.

\begin{proposition}
\label{corollary5}
Let $k$ be an integer in the range $1 < k < g$.  Suppose that $D$ has a chip on $\gamma_k$, $E$ does not, and furthermore, $p_k = p_{k-1} + 1$.  Then $2p_k + ord_{v_k} (\psi_k) = 0$.
\end{proposition}
\begin{proof}
Let $p$ be the point on $\gamma_k$ containing a chip of $D$. Since $p_k = p_{k-1} + 1$, $p$ is a distance of $p_{k-1} m_k$ from $v_k$ in the counterclockwise direction.  Suppose this point lies on the longer arc of $\bar{\gamma_k}$.  The case where $p$ lies on the shorter arc is similar.

Let $d$ be the distance between $v_k$ and $p$ in the counterclockwise direction, and write $x$ for the slope of $\psi$ on this arc.  Then the slope of $\psi$ on the arc from $p$ to $v_{k-1}$ in the counter clockwise direction must be $x-2$, and the slope of $\psi$ on the arc from $v_{k-1}$ to $v_k$ in the counter clockwise direction must be $x + ord_{v_k} (\psi_k)$ (see figure \ref{fig:SlopeD}).

\begin{figure}[!h]
\centering
\includegraphics[width=.25\textwidth]{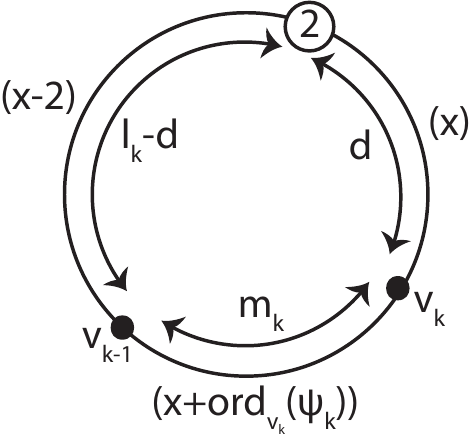}
\caption{Slopes of $\psi$ (outside the loop) and distances (inside the loop) on $\bar{\gamma_k}$}	
\label{fig:SlopeD}
\end{figure}

We have the following equation:
\[ xd + (x-2)(\ell_k-d) + (x + ord_{v_k} (\psi_k))m_k = 0, \]
which gives
\[ \frac{\ell_k}{m_k} (x-2) + 2\frac{d}{m_k} + ord_{v_k} (\psi_k) + x = 0 . \]
Note that $d = p_{k-1} m_k + n(\ell_k + m_k)$ for some non-negative integer $n$, so
\[ \frac{\ell_k}{m_k} (x-2) + 2p_{k-1} + 2n\frac{\ell_k}{m_k} + 2n + ord_{v_k} (\psi_k) + x = 0, \]
or
\begin{equation}
\label{eqn:lk/mk1}
	\frac{\ell_k}{m_k} (x-2+2n) + 2p_{k-1} + 2n + ord_{v_k} (\psi_k) + x = 0.
\end{equation}
Assuming that $x \neq 2 - 2n$, we have:
\[\frac{\ell_k}{m_k} =  \frac{ 2p_{k-1} + 2n + ord_{v_k} (\psi_k) + x}{2-x-2n} . \]

The numerator and the denominator of the right hand side are integers that add up to $2p_{k-1} + ord_{v_k} (\psi_k) +2 = 2p_k + ord_{v_k} (\psi_k)$.  On the other hand, it is clear that $p_k \leq g-1$, so $0 < 2p_k \leq 2g - 2$. Also, by Corollary \ref{corollary1}, $2(k-g) \leq ord_{v_k} (\psi_k) < 0$.  Adding these inequalities gives $2(k-g) < 2p_k + ord_{v_k} (\psi_k) < 2g - 2$, which leads to $\vert 2p_k + ord_{v_k} (\psi_k) \vert < 2g - 2$, contradicting the genericity condition on the edge lengths of $\Gamma_g$.

Hence, $x = 2 - 2n$.  Then equation \ref{eqn:lk/mk1} becomes $2p_{k-1} + 2n + ord_{v_k} (\psi_k) + 2 - 2n= 0$, or $2p_k + ord_{v_k} (\psi_k) = 0$.
\end{proof}

\subsection{Proof of the Main Theorem}

We now continue the proof of our original claim.  As in the proof of Corollary \ref{corollary1}, our main approach is to remove loops from the graph $\Gamma_g$, thereby obtaining a contradiction to the inductive hypothesis.

\begin{proof}[Proof of Theorem \ref{mainthm}]

First, suppose that the lingering lattice path associated to $D$ has a lingering step.  In other words, that $p_k = p_{k-1}$ for some positive integer $k$.  By definition, $D$ has a chip on $\gamma_k$. It follows from Corollary \ref{corollary4} that $k \neq g$.

Remove $\gamma_k$ from $\Gamma_g$ and identify $v_k$ with $v_{k-1}$ to obtain a new graph $\Gamma_{g-1}$ (See Figure \ref{fig:Remove1loop}). We may define a continuous piecewise linear function $\psi'$ on $\Gamma_{g-1}$ such that on every edge of $\Gamma_{g-1}$, the slope of $\psi'$ is the same as the slope of $\psi$. In the remainder of the proof, we will call $\psi'$ the \textit{modified restriction} of $\psi$ to the new graph - in this case $\Gamma_{g-1}$.

\begin{figure}[!h]
\centering
\includegraphics[width=.5\textwidth]{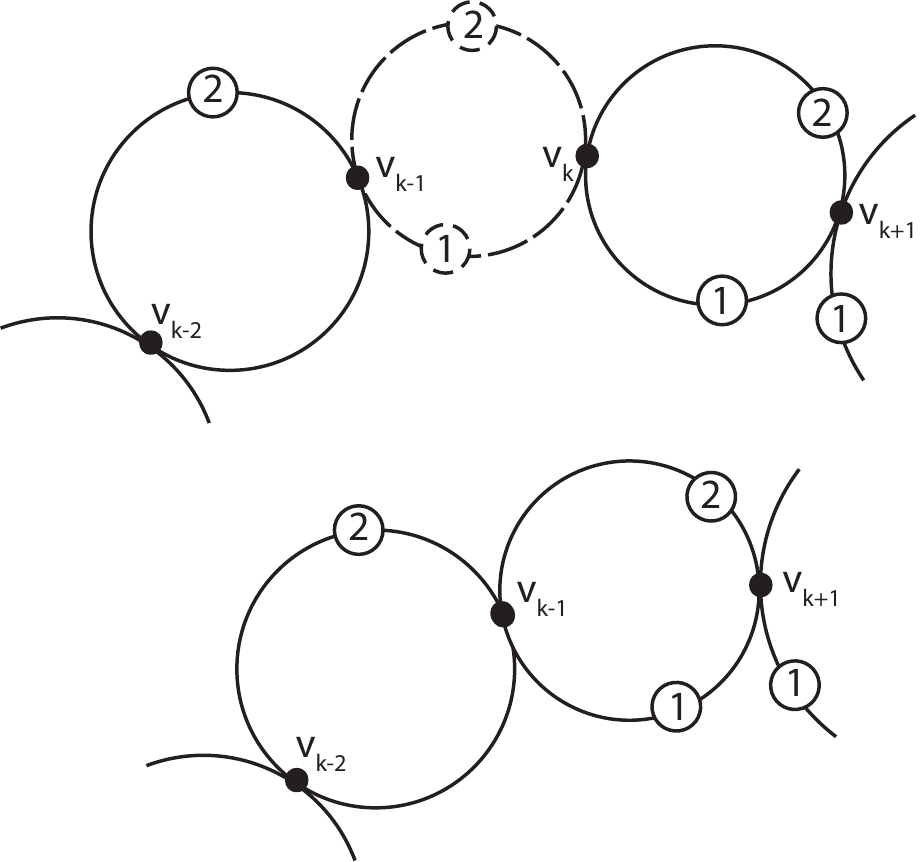}
\caption{Removing $\gamma_k$}
\label{fig:Remove1loop}
\end{figure}

By the definition of $\psi'$, for all $p \in \Gamma_{g-1}$,
\begin{displaymath}
ord_p (\psi') = \left\{ \begin{array}{ll} ord_{v_{k-1}} (\psi_{k-1}) + ord_{v_k} (\psi_{k+1}) & \mbox{if $p = v_{k-1}$;}\\ ord_p (\psi) & \mbox{otherwise.} \end{array} \right.
\end{displaymath}

Therefore, letting $D'$ and $E'$ be the restriction of $D$ and $E$ to $\Gamma_{g-1}$, we see that $div(\psi') = 2D' + E' - K_{\Gamma_{g-1}} + zv_{k-1}$ for some integer $z$.

We have the following equations:
\begin{eqnarray*}
deg \left( 2D' + E' - K_{\Gamma_{g-1}}\right) + z = 0\\
deg \left( 2D + E - K_{\Gamma_g} \right) = 0\\
deg \left(K_{\Gamma_g} \right) - deg \left(K_{\Gamma_{g-1}} \right) = 2\\
deg (2D + E) - deg (2D' + E') = \mbox{$2$ or $3$}
\end{eqnarray*}

It easily follows that $z = \mbox{$0$ or $1$}$. Note that $E'$ is effective, as is $E' + zv_{k-1}$.  On the other hand, note that $p_k = p_{k-1}$, so the lingering lattice path associated to $D'$ is $(p_0, p_1, \ldots , p_{k-1}, p_{k+1}, \ldots , p_g)$.  Thus, $D'$ moves.  It follows that the pair $(D', E' + zv_{k-1})$ on $\Gamma_{g-1}$ contradicts our inductive hypothesis.

It remains to prove the theorem in the case where the lingering lattice path associated to $D$ has no lingering steps.  By Corollaries \ref{corollary3} and \ref{corollary4}, every cell $\gamma_k$ for $k<g$ admits one of the following descriptions:
\begin{enumerate}
\item  (Type-$(E)$ loop)  The cell contains a chip of $E$, but not $D$.
\item  (Type-$(D)$ loop)  The cell contains a chip of $D$, but not $E$.
\item  (Type-$(D,E)$ loop)  The cell contains chips of both $D$ and $E$.
\end{enumerate}
We break this into two cases:

\textbf{Case 1:  There exists a type-$(E)$ loop next to a type-$(D,E)$ loop.}

Suppose that $\gamma_k$ and $\gamma_{k+1}$ form such a pair of loops.  It follows from Corollary \ref{corollary4} that $k < g-1$.  Remove both $\gamma_k$ and $\gamma_{k+1}$ and identify $v_{k+1}$ with $v_{k-1}$ to obtain a new graph $\Gamma_{g-2}$. Let $D', E'$ be the restrictions of $D, E$ to $\Gamma_{g-2}$, and $\psi'$ be the modified restriction of $\psi$ to $\Gamma_{g-2}$. It is clear that $E'$ is effective.

\begin{figure}[!h]
\centering
\includegraphics[width=.6\textwidth]{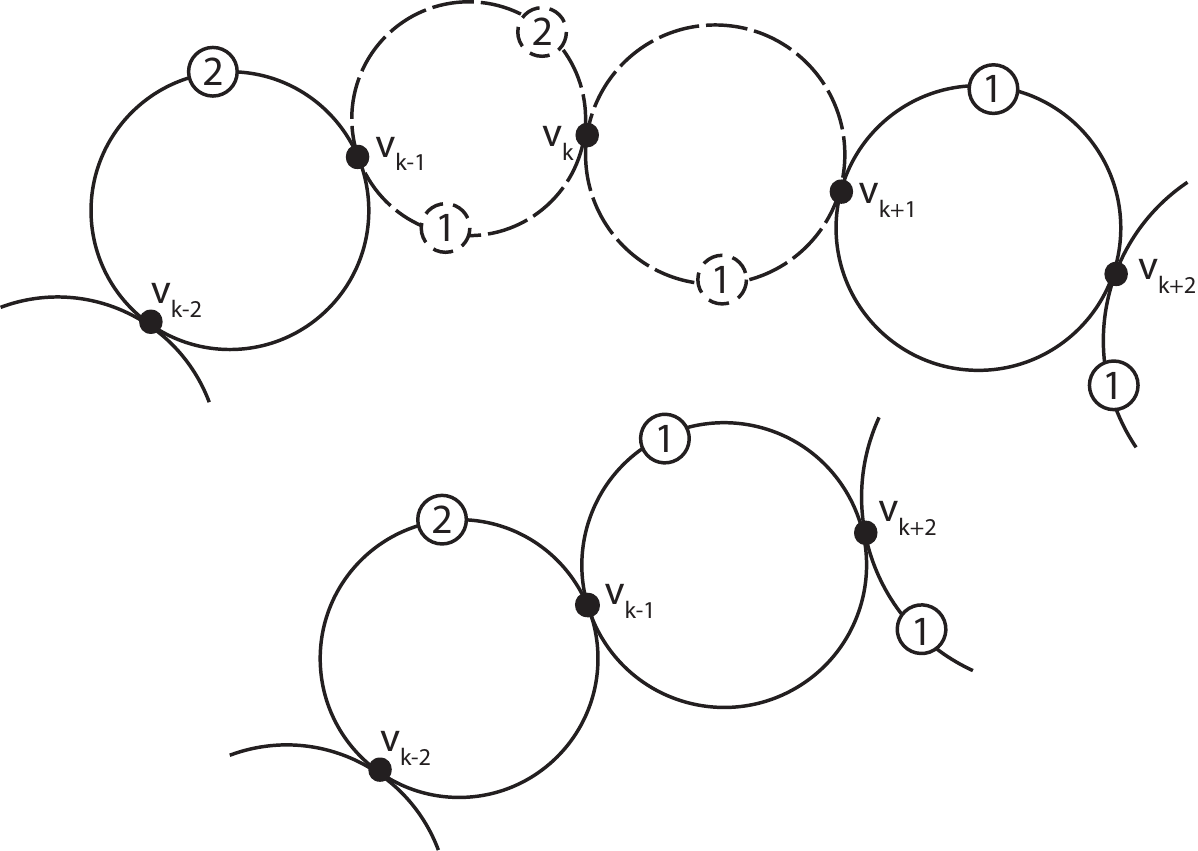}
\caption{Removing $\bar{\gamma_k}$ and $\bar{\gamma_{k+1}}$}
\label{fig:Remove2loops.eps}
\end{figure}

By the definition of $\psi'$, for all $p \in \Gamma_{g-2}$,
\begin{displaymath}
ord_p (\psi') = \left\{ \begin{array}{ll} ord_{v_{k-1}} (\psi_{k-1}) + ord_{v_{k+1}} (\psi_{k+2}) & \mbox{if $p = v_{k-1}$;}\\ ord_p (\psi) & \mbox{otherwise.} \end{array} \right.
\end{displaymath}

In a similar way to the case above, we see that $div(\psi') = 2D' + E' - K_{\Gamma_{g-2}} + zv_{k-1}$ for some integer $z$.  We have the following equations:
\begin{eqnarray*}
deg \left( 2D' + E' - K_{\Gamma_{g-2}}\right) + z = 0\\
deg \left( 2D + E - K_{\Gamma_g} \right) = 0\\
deg \left(K_{\Gamma_g} \right) - deg \left(K_{\Gamma_{g-2}} \right) = 4\\
deg (2D + E) - deg (2D' + E') = 4
\end{eqnarray*}

It easily follows that $z = 0$, which means $2D' + E' \sim K_{\Gamma_{g-2}}$.  On the other hand, note that $p_{k+1} = p_{k-1}$, so the lingering lattice path associated to $D'$ is $(p_0, p_1, \ldots , p_{k-1}, p_{k+2}, \ldots , p_g)$. Thus, $D'$ moves.  As before, we see that the pair $(D', E')$ on $\Gamma_{g-2}$ contradicts our inductive assumption.

\textbf{Case 2:  There is no type-$(E)$ loop next to any type-$(D,E)$ loop.}

If there is no type-$(D)$ loop, then the cells $\gamma_k$ for $k<g$ are either all type-$(E)$ or all type-$(D,E)$. The former gives $p_g = p_0 - g \leq -1$, while the latter gives $deg(2D+E) = p_0 + 3(g-1) > 2(g-1)$,  both of which are impossible.

If after the last type-$(D)$ loop $\gamma_k$, there is no type-$(D,E)$ loop, we have $p_j = p_{j-1} - 1$ for all $j > k$. Hence $p_g = p_k + k - g$.  By Corollary \ref{corollary1} and Proposition \ref{corollary5}, however, we have $2p_k = - ord_{v_k} (\psi_k) \leq 2(g-k)$.  It follows that $p_g \leq 0$, contradicting our assumption that $D$ moves.

Otherwise, the last type-$(D)$ loop $\gamma_k$ is followed by $g-k-1$ type-$(D,E)$ loops $(g-k-1>0)$.  By Lemma \ref{lemma3}, $ord_{v_k} (\psi_k) + 2(g-k) \geq 3(g-k-1)$.  By Proposition \ref{corollary5}, however, $2p_k + ord_{v_k} (\psi_k) = 0$.  It follows that $2p_k \leq k + 3 - g$, but this is impossible because $k < g-1$ and $2p_k \geq 2$.

Therefore, the graph $\Gamma_g$ does not admit a positive-rank divisor $D$ such that $K_{\Gamma_g} - 2D$ is linearly equivalent to an effective divisor.

\end{proof}

\bibliography{TropicalGP}

\end{document}